\documentclass[11pt,a4paper]{article}
\usepackage[DIV=11]{typearea}
\usepackage[small,bf]{titlesec}
\usepackage[english]{babel}
\usepackage[utf8x]{inputenc}
\usepackage{hyperref}
\usepackage{amsmath}
\usepackage{amssymb}
\usepackage{amsfonts}
\usepackage{amsthm}
\usepackage{amscd}
\usepackage{mathdots}
\usepackage[noblocks]{authblk}
\usepackage{verbatim}
\makeatletter
\def\revddots{\mathinner{\mkern1mu\raise\p@
\vbox{\kern7\p@\hbox{.}}\mkern2mu
\raise4\p@\hbox{.}\mkern1mu\raise7\p@\hbox{.}\mkern1mu}}
\makeatother

\theoremstyle{plain}% default
\newtheorem{thm}{Theorem}[section]
\newtheorem{lem}[thm]{Lemma}
\newtheorem{prop}[thm]{Proposition}
\newtheorem{cor}[thm]{Corollary}

\theoremstyle{definition}

\theoremstyle{remark}
\newtheorem*{rem}{Remark}

\newcommand{\R}{\mathbb{R}}
\newcommand{\Q}{\mathbb{Q}}

\newcommand{\Ind}{\mathrm{Ind}}
\newcommand{\Hom}{\mathrm{Hom}}

\newcommand{\SL}{\mathrm{SL}}
\newcommand{\Spin}{\mathrm{Spin}}
\newcommand{\Sp}{\mathrm{Sp}}

\newcommand{\SU}{\mathrm{SU}}
\newcommand{\der}{{\mathrm{der}}}

\newcommand{\Aut}{\operatornamewithlimits{Aut}}

\title{Eisenstein series arising from Jordan algebras} 
\author{Marcela Hanzer\thanks{hanmar@math.hr}}
\affil{Marcela Hanzer, Department of Mathematics, University of Zagreb, Croatia}
\author{Gordan Savin\thanks{savin@math.utah.edu} }
\affil{Gordan Savin, Department of Mathematics, University of Utah,
  Salt Lake City} 
\date{}

\begin{document}

\maketitle

\begin{abstract} We describe poles and the corresponding residual automorphic representations of Eisenstein series attached to maximal 
parabolic subgroups whose unipotent radicals admit Jordan algebra structure. 
\end{abstract}

\section{Introduction}  
Let $G$ be a simple, simply connected algebraic group defined over a global field $k$. Assume that $G$ has a maximal parabolic subgroup $P=MN$ such that $N$ is abelian and  $P$ is conjugated to 
the opposite parabolic $\bar P= M \bar N$. Then $N$ admits  structure of a Jordan algebra $(J,\circ)$. 
The main goal of this article is to study poles and residues of the 
degenerate Eisenstein series  $E(s)$ attached to the parabolic $P$ under an additional assumption that the algebra identity element $e\in J$ can be written as a sum 
$e=e_1 + \cdots + e_r$, for a system of 
perpendicular and absolutely indecomposable idempotent elements $e_i$. 
This assumption allows us to use the technique of Fourier-Jacobi series, due to Ikeda \cite{Ikeda_FJ}, and build an argument inductive on $r$.  Examples of such 
Jordan algebras are $J_r(D)$, the algebras of $r\times r$ hermitian symmetric matrices with coefficients in a composition algebra $D$ over $k$.
In addition,  for $r=2$, there is a class of Jordan algebras $J_2(D)$ parameterized by quadratic spaces $D$ over $k$.  Let $d$ denote the dimension of $D$. 
In order to understand the structure of residual 
automorphic representations, it is necessary to understand the structure of local degenerate principal series representations $I(s)$ attached to $P$ at reducibility points.  
For real groups, in the setting of this paper, this was accomplished by Sahi in a couple of papers, \cite{Sahi_Tube} and \cite{Sahi}. On the other hand, for $p$-adic groups,  
Weissman  \cite{Weissman_small}  analyzes the structure of 
the degenerate principal series representations using a Fourier-Jacobi functor. In a nutshell, this method is a local analogue of Ikeda's method. 

More precisely, the contents of this paper are as follows. In Section 2 we describe the groups and related Jordan algebras. Section 3-5 are devoted to local results. 
 Weissman looks only at the case of 
split, simply laced groups, so in Section 3 we generalize his results to non-split groups. In Section 5 we summarize the results of Sahi in the real case. In order to keep 
the exposition simple,  we assume here that $D$ is either split or totally anisotropic and $d \equiv 0 \pmod{4}$. 
Section 6 is devoted to global results. 
The local Fourier-Jacobi functor works well with Ikeda's method and we combine the two to obtain sharper results. Our final result is a complete description of poles 
and the corresponding residual representations in a natural family of cases, that does not exhaust all possible cases that can be addressed by the methods of this paper. 
If we assume that $d \equiv 0 \pmod{4}$, and some additional minor assumptions that are automatically satisfied if $r>2$, then the Eisenstein series $E(s)$ has simple poles 
at the sequence of odd integers $1, 1+ d/2, \ldots , 1+ (r-1)d/2$, 
and the residual representation is isomorphic to the co-socle of the global induced representation $I(s)$ at the same points. 

Study of Eisenstein series attached to degenerate principal series has a long history, often intertwined with the classical theta correspondences and the Siegel-Weil formula. 
In particular, Ikeda's work deals with symplectic and unitary groups which, in the language of this paper, are the cases $J=J_r(D)$ where $D=k$ or $K$, a quadratic extension of $k$. 
Yamana \cite{Yamana_Siegel_Weil} has taken Ikeda's work further, to quaternionic groups. This works goes beyond the confines of classical groups and is motivated by 
a Siegel-Weil formula in the setting of exceptional theta correspondences, a work in progress of the second author with Wee Teck Gan.

\section{Groups} \label{S:groups} 
Following \cite{Kobayashi_Savin_small},  we shall describe the groups $G$ and Jordan algebras appearing in this paper, starting with split groups. The general case is obtained 
by Galois descent. 
 
 \subsection{Split groups and Jordan algebras} 
 So assume that $G$ is split i.e. it is a Chevalley group.  Let $\mathfrak g$ be the Lie algebra of $G$, and $\Phi$ the root system
arising from a maximal split Cartan subalgebra $\mathfrak t\subseteq \mathfrak g$. In particular, for every $\alpha\in\Phi$, we have the corresponding root space
$\mathfrak g_{\alpha}\subseteq \mathfrak g$. 
 Fix $\Delta=\{\alpha_{1}, \ldots ,\alpha_{l}\}$, a set of simple roots. Now every root can be written as
a sum $\alpha= \sum_{i=0}^l m_i(\alpha)\alpha_i$ for some integers
$m_i(\alpha)$. Every simple root  $\alpha_{j}$ defines a maximal parabolic subalgebra 
$\mathfrak p = \mathfrak p_j =\mathfrak m \oplus \mathfrak n$  where the nilpotent radical $\mathfrak n$ is the direct sum of $\mathfrak g_{\alpha}$ such that $m_{j}(\alpha)>0$. 
Let $\beta$ be the highest root. The algebra
 $\mathfrak n$ is commutative if and only if $m_{j}(\beta)=1$. 
  In the following table we list of all possible pairs 
$(\mathfrak g, \mathfrak m)$ with $\mathfrak n$ commutative and $\mathfrak p$ conjugate to the opposite parabolic by an element in $G$. 
\[
 \begin{array}{c||c|c|c|c|c|c}
 \mathfrak g  & C_n &A_{2n-1} &  D_{2n}  & E_{7} & B_{n+1} & D_{n+1}  \\ 
 \hline
   \mathfrak m^{\der} &A_{n-1} &  A_{n-1}\times A_{n-1} &   A_{2n-1} & E_{6} & B_n &  D_{n}  \\
   \hline 
   \dim \mathfrak n  & n(n+1)/2 &  n^2 &  n(2n-1) & 27 & 2n+1 & 2n  \\
   \hline 
   r & n & n & n & 3 & 2 & 2 \\
   \hline 
   d & 1 & 2 & 4 & 8 & 2n-1 & 2n-2  \\
  \end{array}  
 \] 
 
 The integers $r$ and $d$ are invariants of the (split) Jordan algebra structure $(J, \circ)$ on $\mathfrak n$ which we now describe. 
 The integer $r$ is the cardinality of any maximal set $S=\{ \beta_1, \ldots ,  \beta_r\}$ of strongly orthogonal roots $\alpha$ such that 
  $\mathfrak g_{\alpha} \subseteq \mathfrak n$.  Observe that $r=1$ if and only if $G=\SL_2$. 
   For every $\beta_i\in S$ 
take an $\mathfrak {sl}_2$-triple $(f_i, h_i, e_i)$  where $e_i\in \mathfrak g_{\beta_i}$ and  $f_i\in \mathfrak g_{-\beta_i}$. Let 
\[ 
f=\sum_{i=1}^r f_i, \, h=\sum_{i=1}^r h_i \text{ and }  e=\sum_{i=1}^r e_i. 
\] 
Since the roots $\beta_i$ are strongly orthogonal, $(f,h,e)$ is also an $\mathfrak {sl}_2$-triple. 
The semi-simple element $h$ preserves the decomposition 
\[ 
\mathfrak g = \bar{\mathfrak n}\oplus \mathfrak m \oplus \mathfrak n. 
\] 
More precisely,  $[h,x]=-2 x$ for all $x\in \bar{\mathfrak n}$, $[h,x]=0$ for all $x\in {\mathfrak m}$, and 
$[h,x]=2 x$ for all $x\in \mathfrak n$. The triple $(f, h, e)$  lifts to a homomorphism $ \varphi : \SL_2 \rightarrow G$.  
The element  
\begin{equation} \label{E:special_element}
w_0=
 \varphi \left(
\begin{matrix} 0 & 1 \\ 
-1&  0 \end{matrix}
\right) 
\end{equation} 
normalizes $M$ and conjugates $\mathfrak n$ into $\bar{\mathfrak n}$, and vice versa. 
The Jordan algebra  multiplication $\circ$ on $J=\mathfrak n$ is defined by 
\[
x\circ y = \frac{1}{2}[ x, [f,y]]. 
\]
Note that $e$ is the identity element. 
The elements $e_i$ are mutually perpendicular ($e_i \circ e_j =0$ if $i\neq j$) and idempotent ($e_i\circ e_i=e_i$) elements in $J$ such that 
$e_1+ \cdots + e_r=e$. These idempotent elements give a Pierce decomposition of $J$, 
\[ 
J = \bigoplus_{1\leq i \leq r} J_{ii} \oplus \bigoplus_{1\leq i<j\leq r} J_{ij} 
\] 
where   
\[ 
J_{ii} = \{ x\in J ~|~ e_i \circ x = x\}
\] 
and 
\[ 
J_{ij} = \{ x\in J ~|~ e_i \circ x =\frac{1}{2}  x \text{ and } e_j \circ x= \frac{1}{2} x \}. 
\] 
The space $J_{ii}$ is one-dimensional and spanned by $e_i$. The dimension of $J_{ij}$, for $i<j$, is $d$. It is independent of $i<j$. Let $D=J_{12}$. 
Then $D$ is a quadratic space with a split quadratic form 
\[ 
q(x)= \frac{1}{2}\kappa([f_1, x], [f_2,x]]), 
\] 
where $\kappa(\cdot, \cdot )$ is the Killing form normalized by $\kappa(f_1, e_1 )=1$. If $r>2$, then one can identify all $J_{ij}$ with $D$ and, using 
$J_{12}\circ J_{23} \subset J_{23}$,  
endow $D$ with a multiplication such that $q(xy)=q(x)q(y)$, for all $x,y\in D$, i.e. $D$ is a composition algebra. 

\smallskip 

 Each triple $(f_i, h_i, e_i)$  lifts to a homomorphism of algebraic groups 
$\varphi_i : \SL_2 \rightarrow G$. 
By restricting $\varphi_i$ to the torus of diagonal matrices in $\SL_2$ we obtain a homomorphism 
(a co-character)  $\omega^{\vee}_i : \mathbb G_m \rightarrow M$, 
\begin{equation}  \label{E:co-character} 
\omega^{\vee}_i(t) = \varphi_i \left(
\begin{matrix} t & 0 \\ 
0& t^{-1} \end{matrix}
\right). 
\end{equation} 
Let $T_r\subseteq M$ be the torus generated by all  $\omega_i^{\vee}(t)$. 
Any element in $T_r(k)$ is uniquely written as a product of $\omega_i^{\vee}(t_i)$ for some $t_i\in k^{\times}$. 
One checks that the restricted root system with respect to $T_r$ is of the type $C_r$.
Since $G$ is simply connected, the group of characters $\Hom(M, \mathbb G_m)\cong \mathbb Z$ has a canonical generator $\omega$ which, 
when restricted to the torus $T$,  is the fundamental weight corresponding to the simple root $\alpha_j$, defining $M$. 
Moreover, the kernel of $\omega$ is $M^{\der}$, the derived group of $M$. A simple computation with the root data shows that 
\begin{equation} \label{E:co-root}
\omega(\omega^{\vee}_i(t))=t. 
\end{equation}

 \subsection{Fourier-Jacobi tower} 
 Let $Q=LV$ be the standard parabolic subgroup of $G$ such that the roots of $L^{\der}$ are perpendicular to the highest root $\beta$. In other words, the Lie algebra of 
 $L^{\der}$ is the centralizer in $\mathfrak g$ of the $\mathfrak{sl}_2$-triple corresponding to $\beta$. 
 The unipotent radical $V$ is a Heisenberg group with the center $Z=\exp(\mathfrak g_{\beta})$.  The group $L^{\der}$ is semi-simple and, by inspection,  it has a 
 unique simple factor $G_1$ that is not contained in $M$. 
 Let $M_1= G_1\cap M$ and $N_1=G_1\cap N$. Then $P_1=M_1N_1$ is maximal parabolic subgroup of $G_1$.  
 The unipotent radical $N_1$ has a Jordan algebra $J_1$ structure that is easily related to $J$. Indeed, if we pick the set $S=\{ \beta_1, \ldots ,  \beta_r\}$ of 
 strongly orthogonal roots such that 
 $\beta_1=\beta$, then $J_1$ is the sum of the pieces in the Pierce decomposition where all the indexes are greater than 1. 
 This process can be continued, and will give a sequence of simple groups $G, G_1, \ldots G_{r-1}\cong \SL_2$ 
 and maximal parabolic groups  with unipotent radicals $N, N_1, \ldots N_{r-1}\cong k$. Moreover, this process gives us a canonical choice of $S$ such that $\beta_i$ 
 is the the highest root of $G_{i-1}$.
 This sequence is the Fourier-Jacobi tower referred to in the 
 title, and the last group $\SL_2$ is the terminal group. This process is summarized by the following table: 
 
 \[
 \begin{array}{c||c|c|c|c|c|c}
 \mathfrak g  & C_n &A_{2n-1} &  D_{2n}  & E_{7} & B_{n+1} & D_{n+1}  \\ 
 \hline
   \mathfrak m^{\der} &A_{n-1} &  A_{n-1}\times A_{n-1} &   A_{2n-1} & E_{6} & B_n &  D_{n}  \\
   \hline 
    \mathfrak l^{\der}  & C_{n-1}  &  A_{2n-3} &  A_1 \times D_{2n-2}  & D_6 &  A_1 \times B_{n-1} &  A_1 \times D_{n-1} \\
   \hline 
   \mathfrak g_1  & C_{n-1} & A_{2n-3} & A_1 & D_6  & A_1 & A_1 \\
   \hline 
   \mathfrak m_1^{\der}  & A_{n-2}  & A_{n-2}\times A_{n-2} & 0  & D_5 & 0 & 0 \\
  \end{array}  
 \] 
 We shall need the following remark. 
Let $\varphi : \SL_2 \rightarrow G$ arising from this $S$. Then $w_0$, defined by the equation (\ref{E:special_element}), permutes the simple roots of $M$.

 \subsection{Non-split groups} 
In \cite{Kobayashi_Savin_small} it is proved that the centralizer of $\varphi(\SL_2)$ in $\Aut(G)$ is precisely $\Aut(J)$. 
Therefore, by functoriality of Galois cohomology,  a class $c\in H^1(k, \Aut(J))$ defines a class 
in $H^1(k, \Aut(G))$. Hence the class $c$ defines a Jordan algebra $J_c$, a form of $J$, and a form $G_c$ of $G$ whose Lie algebra contains the triple $(f, h, e)$. 
Hence $G_c$ contains a form $P_c$ of $P$ whose unipotent radical is isomorphic to $J_c$. Moreover, if the form $J_c$  arises from a form of $D$, 
i.e. $J_c$ contains the absolutely indecomposable idempotents $e_i$,  then the Lie algebra 
of $G_c$ contains the triples $(f_i, h_i, e_i)$, and $G_c$ contains the split 
torus $T_r$, defined above. This torus is maximal if $D$ is anisotropic. 

Henceforth we shall omit the subscript $c$, and $G$ will denote a group arising from a Jordan algebra $J\cong J_r(D)$ where $D$ is a composition algebra if $r\geq 3$, 
and simply a quadratic space if $r=2$. In particular, $G$ contains the maximal parabolic subgroup $P=MN$ such that $N\cong J$, and the Heisenberg parabolic 
subgroup $Q=LV$, whose center is $Z\cong J_{11}$, and $G$ is the first term of a Fourier-Jacobi tower where the next group is $G_1$ with the the maximal parabolic subgroup 
$P_1=M_1N_1$ such that $N_1\cong J_1$ where $J_1\cong J_{r-1}(D)$ etc.

\section{Representations of $p$-adic groups} 

In this section $k$ is a $p$-adic field. The goal of this section is to extend the results of Weissman in \cite{Weissman_small} to non-split groups. 

\subsection{Fourier-Jacobi functor} 
We fix a non-trivial additive character $\psi$ of $Z\cong k$, the center of the unipotent radical $V$ of the parabolic $Q=LV$. 
 Let $\omega_{\psi}$ be the corresponding irreducible representation of $V$ with the central character 
   $\psi$. Note that $L^{\der}=[L,L]$ (or its 2-fold cover) acts on $\omega_{\psi}$, via the Weil representation. 
    Let $\pi$ be a smooth representation of $G$ and $\pi_{Z,\psi}$ the maximal quotient of $\pi$ on which $Z$ acts as $\psi$.  Then $\pi_{Z,\psi}$ is a multiple of $\omega_{\psi}$, and 
    \[ 
   FJ(\pi)  = \Hom_V(\omega_{\psi}, \pi_{Z,\psi}) 
    \] 
    is naturally a $L^{\der}$-module. 
   The Fourier-Jacobi functor $\pi\mapsto FJ(\pi)$ is exact \cite{Weissman_small}. 
   
   \smallskip

   Let $\omega$ be the  character of $M$, introduced previously, giving the isomorphism of $M/M^{\der} \cong \mathbb G_m$. 
    Let $\chi$ be a quadratic character and $|\cdot |^s$ the absolute value character, taken to the power $s\in \mathbb C$,  of  $k^{\times}=\mathbb G_m(k)$. 
   We can pull back these two characters to $M$ via $\omega$.  Let $I(\chi, s)=\Ind_P^G(\chi \otimes |\cdot |^s)$ be the degenerate principal series representation of $G$. 
   The modular character $\rho_P$ can be expressed in terms of $\omega$ using the relation  $\omega(\omega_1^{\vee}(t))=t$. The conjugation action of 
   $\omega_1^{\vee}$ on $N\cong J_r(D)$ is given by multiplication by $t^2$ on $J_{11}\cong k$, and by $t$ on each $J_{1,i}\cong D$, for $1< i \leq r$. Thus 
   $\rho_P(m) = |\omega(m)|^{2+ (r-1)d}$, where $d=\dim D$. The trivial representation is a quotient of $I(1, 1+ (r-1)\frac{d}{2})$. 
   
   \smallskip

   We shall now compute the action of the Fourier-Jacobi functor on $I(\chi,s)$. 
   In order to state the result we need some additional data arising from the Weil representation 
  $\omega^D_{\psi}$  of $\SL_2(k)$ on $C_0^{\infty}(D)$. For every $t\in k^{\times}$
   let $h(t)$ be the element in the universal central extension of $\SL_2(k)$, introduced by Steinberg, 
   projecting to $\left(\begin{smallmatrix} t & 0 \\ 0& t^{-1} \end{smallmatrix}\right)$.  In particular,
   $h(t) h(s)= h(ts) (t,s)$, where $(t,s)$ is the Steinberg symbol.  Then, for every $f\in C_0^{\infty}(D)$, 
   \[ 
   \omega^D_{\psi}(h(t))(f)(x) = \chi_D(t) |t|^{d/2} f(tx) 
   \] 
   where $\chi_D(t)$ is a fourth root of 1. If $\dim D$ is odd, then $\chi_D(t) \chi_D(s)=\chi_D(ts) (t,s)_2$ where $(t,s)_2$ is the Hilbert symbol. If $\dim D=2n$, then 
   $\chi_D$ is independent of $\psi$. It is a quadratic character of $k^{\times}$ that corresponds to the quadratic algebra 
   $K=k(\sqrt{\Delta})$, by the local class field theory, where $d$ is the discriminant of the quadratic form $q$ on $D$.
    More precisely, if $q=a_1 x_1^2 + \ldots + a_{2n} x_{2n}^2$, then $\Delta= (-1)^n a_1 \cdot \ldots \cdot a_{2n}$. If 
    $(D,q)$ is a direct sum of $n$-hyperbolic planes, or if the anisotropic kernel of $(D,q)$ is a quaternion algebra, then $\chi_D$ is trivial.
    However, if the anisotropic kernel of $(D,q)$ is a quadratic field $K$, then $\chi_D$ corresponds to $K$ by the local class field theory. 
    
    \begin{thm}  
    \label{FJ_local_degenerate_principal}
    Let $I(\chi,s)$ be the degenerate principal series corresponding to the pair $(G,P)$ where the unipotent radical of $P$ is isomorphic to 
    the Jordan algebra $J_r(D)$. Let $(G_1, P_1)$ be the next pair in the Fourier-Jacobi tower. Then 
    \[ 
    FJ(I(\chi,s))= I_1(\chi\chi_D,s). 
    \] 
    \end{thm}
    \begin{proof} 
    This is proved by  Weissman   in \cite{Weissman_small} if $G$ is split and simply laced. 
    The proof given there extends easily to the more general class of groups considered here. 
    (Note that  $G$ is split and simply laced precisely when $D$, considered as a quadratic space, is isomorphic to a direct sum of hyperbolic planes). 
    \end{proof} 
    
    \subsection{Decomposition of degenerate principal series} 
    
    Let $\pi$ be an irreducible representation of $G$. If $\pi$ is not the trivial representation, then there exists a non-trivial character $\psi$ of $Z$ such that 
    $\pi_{Z,\psi}\neq 0$. The group $L$ acts on $Z$ by conjugation and, therefore, on non-trivial characters of $Z$. If this action is transitive then, without loss of generality, 
    we can fix a non-trivial character $\psi$ of $Z$ and then $\pi$ is non-trivial if and only if $FJ(\pi)\neq 0$. In this case 
    Theorem \ref{FJ_local_degenerate_principal} can be used to compute points of reducibilities of the degenerate principal series $I(\chi,s)$ using the induction on $r$.

    Since the Pontrjagin dual of $Z$ is  isomorphic to $Z$, over local fields, the group $L$ acts transitively on non-trivial characters
     if and only if acts transitively on non-trivial elements in $Z$. 
     This is true over the algebraic closure of $k$ and it holds over $k$ if $H^1(k, C)$ is trivial, where $C$ is the centralizer of 
    $e_1$ in $L$.  (We use $Z\cong J_{11}=k\cdot e_1$.) If $G$ is not of the absolute type $C_{2r}$ or $A_{2r-1}$, then 
    $C=L^{\der}$. Since $k$ is $p$-adic then the Galois cohomology of simply-connected groups is trivial, and the transitivity holds. It also holds for $G=\SL_{2r}$. 
     Hence it fails only  if $G=\Sp_{2r}$ or $\SU_{2r}$ i.e. $J=J_r(k)$ or $J_r(K)$, respectively, where $K$ is a quadratic field extension of $k$. 
      In these two cases $C/L^{\der}\cong \mu_2$ and $K^1$, respectively, where $K^1$ is the group 
     of elements of norm one in $K^{\times}$. Thus the orbits of non-trivial characters are parameterized by the classes of squares in 
     $k^{\times}$ and $k^{\times}/N_{K/k}(K^{\times})$, respectively. 
      
      Thus, in the following Theorem \ref{comp_series_r=2_p-adic} and Theorem \ref{comp_series_r>2_p-adic} we can use  Theorem \ref{FJ_local_degenerate_principal} to get the length of the degenerate principal series. 
     \begin{thm} 
     \label{comp_series_r=2_p-adic}
     Let $I(\chi,s)$ be the principal series of $G$ arising from the maximal parabolic subgroup $P$ whose radical $N$ is isomorphic to the 
     Jordan algebra $J_2(D)$ where $\dim D>2$. Assume that $\chi$ is a quadratic character and $s$ real. 
     \begin{enumerate} 
     \item $\dim D=2n-2$, and the discriminant of the quadratic form $q$ is trivial, i.e. $\chi_D=1$. If $\chi\neq 1$ then $I(\chi,s)$ is irreducible unless 
     $s=0$ and then it is a direct sum of two non-isomorphic irreducible representations. If $\chi=1$ then $I(\chi,s)$ is irreducible unless 
     $s=\pm 1, \pm n$ and then it has a non-split composition series of two non-isomorphic irreducible representations.
     \item $\dim D=2n-2$, and the discriminant is non-trivial, i.e. $\chi_D=\chi_K$ where $K$ is a quadratic field extension of $k$. 
      If $\chi\neq \chi_K$ then $I(\chi,s)$ is irreducible unless 
     $s=0$ and then it is a direct sum of two non-isomorphic irreducible representations, or $\chi=1$ and $s=\pm n$ where the trivial representation occurs. 
     If $\chi=\chi_K$ then $I(\chi,s)$ is irreducible unless 
     $s=\pm 1$ and then it has a non-split composition series of two non-isomorphic irreducible representations.
     \item $\dim D=2n-1$. Then $I(\chi,s)$ is irreducible unless 
     $s=\pm 1/2$ and then it has a non-split composition series of two non-isomorphic irreducible representations, or $\chi=1$ and $ s=\pm(n+1/2)$ where the trivial representation occurs.
     \end{enumerate} 
     \end{thm} 
     \begin{proof}  We shall use the Fourier-Jacobi functor, note that $G_1=\SL_2(k)$ (or its two-fold cover) in  all three cases, since $r=2$. Consider the first case. 
    The location of the trivial representation is at $\chi=1$ and $s=\pm n$, as previously noticed. 
      The Fourier-Jacobi functor is exact, takes $I(\chi,s)$ to $I_1(\chi,s)$ and kills only the trivial representation. Hence, if $(\chi,s) \neq (1,\pm n)$, then a non-trivial 
      composition series in $I(\chi,s)$ will give one for $I_1(\chi, s)$. The decomposition of the principal series of $\SL_2(k)$ is well known.
       It follows that $I(\chi,s)$ reduces (possibly) only for $\chi=1$ and $s=\pm 1$, or $\chi\neq 1$ and $s=0$.
       Irreducibility of $I(0,\chi)$ implies existence of the complementary series, 
       which must end before the points where the trivial representation is contained. This forces reducibility for $\chi\neq 1, s=0$ and $1, s=\pm 1$. 
       The parts (2) and (3) are proved similarly, for (3) one uses that the principal series of $\widetilde{\SL}_2(k)$ reduces at $s=\pm 1/2$. 
           
      \end{proof}

      \begin{thm} 
      \label{comp_series_r>2_p-adic}
     Let $I(\chi,s)$ be the principal series of $G$ arising from the maximal parabolic subgroup $P$ whose radical $N$ is isomorphic to the 
     Jordan algebra $J_r(D)$ such that $\chi_D=1$.  Let $d=\dim D$. Assume that $\chi$ is quadratic and $s$ real. If $\chi\neq 1$ then $I(\chi,s)$ is irreducible unless $s=0$, 
     and then it is a direct sum of two non-isomorphic representations. If $\chi=1$, then then $I(\chi,s)$ is irreducible unless 
     $s= \pm 1, \pm(1+d/2), \ldots, \pm(1+ (r-1) d/2)$ and then it has a non-split composition series of two non-isomorphic irreducible representations.
    \end{thm} 
    \begin{proof}  In view of the previous theorem, we can assume that $r\geq 3$. Hence $D$ is a composition algebra. Moreover, the condition $\chi_D=1$ implies that 
    $D$ is either split  (and even dimensional) or a quaternion algebra. In the former case $G$ is split and simply laced, so this case was covered by Weissman. 
    Hence it remains to do the case when $D$ is a quaternion algebra. 
    This is proved by induction on $r$.  Assuming the result for $r-1$, 
    the Fourier-Jacobi functor implies that reducibility points are possibly only those listed and the length of the composition series is not longer than two. 
    In the next section we shall show that,  for $\chi=1$, the spherical representation is a proper subquotient at the indicated points, see Corollary \ref{C:quaternionic}. 
     \end{proof} 
 
   \section{Intertwining operators} 
   We collect some facts we need about principal series representations. In this section $k$ is a local field. 
   
   \subsection{$c$--function for split groups} 
   Here we assume that $G$ is split and simply laced, i.e. $D$ is a sum of hyperbolic planes, 
   and let $B$ be the Borel subgroup, corresponding to our choice of simple roots.
    Let $\chi$ be an unramified character of $T$, for $w$ in the Weyl group of $G$ we have standard local intertwining operators $A(\chi,w)$ 
\[A(\chi,w):\mathrm{Ind}_{B}^{G}(\chi)\to \mathrm{Ind}_{B}^{G}(w(\chi)).\]
In this definition  the choice of the Haar measure on $k$ is such that, in the $p$-adic case, the measure of the ring of integers is 1. 
Let $f_{\chi}\in I(\chi)$ be the unique spherical vector normalized so that $f_{\chi}(1)=1$. Then $A(\chi,w) (f_\chi)= c(\chi,w) f_{w(\chi)}$ where 
the factor for $c(\chi,w)$ is given by the Gindikin-Karpelevič formula (cf. \cite{Gindikin-Karpelevich} for archimedean fields 
   and \cite{Langlands_Euler} for $p$-adic fields)
\begin{equation}
\label{norm_factors_defn}
c(\chi,w)=\prod_{\substack{\alpha \in \Phi^+,\\ w(\alpha)<0}}\frac{L(0,\chi\circ \alpha^{\vee})}{L(1,\chi \circ \alpha^{\vee})},
\end{equation}
where $\alpha^{\vee}$ is the co-root corresponding to a root $\alpha$ and $L$-functions are Tate's $L$-functions. 
 If $\chi\circ \alpha^{\vee} = |\cdot |^s$, then $L(0,\chi \circ \alpha^{\vee}) =\zeta(s)$ and $L(1,\chi \circ \alpha^{\vee}) =\zeta(s+1)$ where 
\[ 
\zeta(s) = (1-q^{-s})^{-1}
\] 
if $k$ is a $p$-adic field with the residual field of order $q$, and 
\[ 
\zeta(s)= \pi^{-\frac{s}{2}} \Gamma(s/2) 
\] 
if $k\cong \mathbb R$. We shall use this
formula to determine the action of the standard intertwining operator $A(s) : I(s) \rightarrow I(-s)$ 
\[ 
A(s)(f)(g)= \int_N  f(w_0 n g) ~dn 
\] 
on the spherical vector, where $w_0$ is the element defined by the equation (\ref{E:special_element}).
 It permutes simple roots of $M$ and maps the roots that span $N$ to the roots that span $\bar N$. Concretely, it is the product of the longest Weyl group elements of $G$ and $M$. 
Let $f_s\in I(s)$ be the spherical vector normalized by $f_s(1)=1$.  Let $\chi_s$ be an unramified character of $T$ such that $I(s)$ is a subrepresentation of 
$\mathrm{Ind}_{B}^{G}(\chi_s)$. Then $A(s)(f_s) =c(\chi_s, w_0) f_{-s}$, which reduces the computation to a combinatorial exercise. We summarize the result  
in the following lemma:

\begin{lem} \label{L:local_c} 
Let $I(s)$ be the degenerate principal series for split, simply laced $G$ arising from a parabolic $P=MN$ such that $N\cong J_r(D)$. Let $d=\dim D$. Let 
$f_s\in I(s)$ be the normalized spherical vector, and $c(s)$ the complex function defined by $A(s)(f_s) =c(s) f_{-s}$. 
  Then 
\[ 
c(s) =\prod_{i=0}^{r-1} \frac{\zeta(s- i d/2)}{\zeta(s +i d/2 +1)}
\] 
where $\zeta(s)$ is the local zeta function as above. 
\end{lem} 

Now one can easily understand the poles and zeroes of $c(s)$, using the poles of $\zeta(s)$. 
 In the $p$-adic case,  for $s$ real, $\zeta(s)$ never vanishes and has a simple pole at $s=0$.  It follows that $c(s)$ has simple poles 
 at $s=0, d/2, \ldots , (r-1)d/2$ and simple zeros at $s=-1, -1- d/2 , \ldots, -1 - (r-1)d/2$. Thus,  at these points,  $I(s)$ has a 
 composition series of length two and the spherical representation is a unique irreducible submodule. 
  In the real case, $\zeta(s)$ has a simple pole at negative even integers. 
  Thus,  if $d$ is divisible by 4, $c(s)$ has zeros at negative odd integers and poles at even positive integers.  
  We summarize: 
  %More precisely, at the points of interest  $s=-1, -1- d/2 , \ldots, -1 - (r-1)d/2$,  the function $c(s)$ has a zero of order 
 %$1, 2,\ldots , r$, respectively, and this is precisely the length of the Jantzen filtration at these points. [Sahi?] The spherical representation is a unique irreducible  submodule. 

%\begin{cor} \label{C:spherical_split} 
% If $G$ is split and $d$ is even, then the spherical vector generates a proper submodule of $I(s)$ for $s=-1, -1- d/2 , \ldots 1 - (r-1)d/2$. 
%\end{cor} 
%\begin{proof} 
%As noted above, for these $s$ the function $c(s)$ has a zero.  In particular,  the standard intertwining operator (which is always non-zero) kills the spherical vector in $I(-s)$. 
%\end{proof} 

\begin{cor} \label{C:non-vanishing_split} 
 If $G$ is split, the local field $k$ is $p$-adic or real, and $d\equiv 0 \pmod{4}$ then $c(s)$ is not vanishing at odd positive integers, in particular, at 
$s=1, 1+ d/2 , \ldots , 1 + (r-1)d/2$. 
\end{cor} 

For split $p$-adic groups, the Satake parameters of the irreducible spherical quotients, at positive reducibility points, have a nice description. 
As previously, let $S=\{\beta_1, \ldots, \beta_r\}$ 
be a maximal set of strongly orthogonal roots spanning $\bar N$. Let $\varphi_i :\SL_2 \rightarrow G$ be the homomorphism 
corresponding to $\beta_i$, for every $i$.  For $j=2, \ldots, r$, let $\psi_j : \SL_2 \rightarrow G$, be the homomorphism given by the product of 
$\varphi_1, \ldots, \varphi_{j-1}$, and $\psi_1$ is the trivial homomorphism. 
(The actual choice of $S$ is not important, since for different choices of $S$ resulting $\psi_j$ are $G$-conjugated.)
The corresponding unipotent class in $G$ is $(j-1)A_1$ in the Bala-Carter notation. 
 Let $\hat G(\mathbb C)$ be the Langlands dual group.  
 Let $\hat\psi_j : \SL_2 \rightarrow \hat G(\mathbb C)$ be the homomorphism that corresponds to $\psi_j$ via the Spaltenstein order reversing map from 
 unipotent orbits of $G$ to unipotent orbits of $\hat G(\mathbb C)$ \cite{Carter}. 
Then the Satake parameter of the spherical quotient of $I(s)$ at the reducibility point $s_0=1+(r-j)d/2$ is 
\[ 
 \hat\psi_j \left(
\begin{matrix} q^{1/2} & 0 \\ 
0 & q^{-1/2} \end{matrix}
\right). 
\] 
If $ d\equiv 0 \pmod{4}$ then each $\hat\psi_j$ corresponds to a distinguished unipotent orbit in $\hat G(\mathbb C)$ i.e. one that 
it is not contained in a proper Levi. 
 It follows that the Aubert dual of the spherical representation is a square integrable representation. We record: 

\begin{prop} \label{P:aubert_dual} 
 If $G$ is split, the local field $k$ is $p$-adic or real, and $d\equiv 0 \pmod{4}$ then the spherical quotients at 
$s=1, 1+ d/2 , \ldots , 1 + (r-1)d/2$ are Aubert duals of square integrable representations. 
\end{prop}

\subsection{$c$--function for non-split groups}

In this section $k$ is a $p$-adic field with the residual field of order $q$ and  $D$ is a  quaternion algebra over $k$. 
 Let $O$ be the maximal order in $D$ and $\pi$ a prime element in $O$. For every 
  $x\in D$ let $|x|$ be the reduced norm composed with the usual absolute value on $k$. In particular $|\pi|=q^{-1}$.  
 Let $I(s)$ be the principal series for $\SL_2(D)$ defined as the set of all smooth functions on $\SL_2(D)$ such that 
 \[ 
 f( ( \begin{smallmatrix} a & b \\ 0 & c\end{smallmatrix}) g)= |a/c|^{\frac{s}{2} + 1} f(g) 
 \]
  for all choices of data.  Consider the intertwining map $A(s) : I(s) \rightarrow I(-s)$ defined by 
 \[ 
 A(s)(f)(g) =\int_{D} f( ( \begin{smallmatrix} 0 & 1 \\ -1 & 0\end{smallmatrix}) ( \begin{smallmatrix} 1 & x \\ 0 & 1\end{smallmatrix}) g) ~dx 
 \] 
 where $dx$ is the invariant measure on $D$ normalized so that the volume of $O$ is 1.

 \begin{lem} Let $f_s\in I(s)$ be the unique $\SL_2(O)$-invariant function such that $f_s(1)=1$, and let $c(s)$ be the function defined by $M(s)(f_s)=c(s) f_{-s}$. Then 
 \[ 
 c(s) = \frac{\zeta(s)}{\zeta(s+2)}. 
 \]  
 where $\zeta(s) = (1-q^{-s})^{-1}$. 
 \end{lem}

 \begin{proof} This is surely well known, but we include a short proof for convenience. The value $c(s)$ is equal to $M(s) f_s(1)$. 
Write $D$ as a union 
 \[ 
 O \cup \pi^{-1} (O\setminus (\pi)) \cup \pi^{-2} (O\setminus (\pi)) \cup \ldots. 
 \] 
 The function $x\mapsto f( ( \begin{smallmatrix} 0 & 1 \\ -1 & 0\end{smallmatrix}) ( \begin{smallmatrix} 1 & x \\ 0 & 1\end{smallmatrix}))$ is equal to $1$ on $O$ and to 
 $q^{-n(s+2)}$ on $\pi^{-n} (O\setminus (\pi))$. This leads us to the sum 
 \[ 
 c(s)=1+ (q^2-1)q^{-(s+2)} + (q^4-q^2) q^{-2(s+2)} + (q^6-q^4) q^{-3(s+2)} + \ldots  
 \] 
 which can be easily summed up to give the claimed result.
  \end{proof} 
 
 Now we can compute the $c$-function for the degenerate principal series $I(s)$ for the group $G$ corresponding to the Jordan algebra $J_r(D)$, by factoring the standard 
 intertwining map $A(s): I(s) \rightarrow I(-s)$ as a product of intertwining maps corresponding to simple root reflections in the restricted root system (of type $C_r$). We summarize 
 the computation in the following lemma:

  \begin{lem}
\label{c_function_quaternionic}
 Assume that $G$ corresponds to $J_r(D)$ where $D$ is a quaternion algebra. Let $c(s)$ be the function such that 
$A(s) f_s= c(s) f_{-s}$ where $A(s): I(s) \rightarrow I(-s)$ is the standard intertwining operator. Then, up to a non-zero constant, 
\[ 
c(s) =\prod_{i=0}^{r-1} \frac{\zeta(s- 2i)}{\zeta(s \pm (2i +1))}
\] 
where the signs in the denominator alternate in $i$, so that the sign is $+$ for $i=r-1$. In words, $c(s)$ has a simple pole at even integers 
$0,2, \ldots, 2(r-1)$ and a simple zero at odd integers $-1-2(r-1), 1+2(r-2), -1 - 2(r-3), ...$ 
\end{lem}

\begin{cor}  \label{C:quaternionic} 
If $G$ corresponds to $J_r(D)$ where $D$ is a quaternion algebra, then the spherical vector generates 
a proper submodule of $I(s)$ for 
\[ 
s= -1-2(r-1), 1+2(r-2), -1 - 2(r-3), ....
\]  
\end{cor} 
\begin{proof} 
The standard intertwining operator is always non-zero, and at these points the $c$-function vanishes. 
\end{proof} 

\begin{rem}
In particular, the quotient at $s= 1+ (r-2)$ (a minimal representation)  is not spherical.  This agrees with results of Gan and Savin \cite{Gan_Savin} where it is observed 
that minimal representations of tame non-quasi split groups are not spherical. 
\end{rem}

   \section{Real groups} 

In this section we assume that $k$ is a real or complex field and $J=J_r(D)$  and $d=\dim D$ is even. If $k=\mathbb R$ then $D$ is assumed to be either split or an anisotropic quadratic space of dimension divisible by 4. (Either of this condition will assure that $\chi_D=1$.) If $D$ and hence $G$ are split, then the principal series $I(s)$ reduces at the 
points $1, 1+ d/2, \ldots 1+ (r-1) d/2$ and the spherical representation is the unique irreducible quotient \cite{Sahi}. 
 
\smallskip 
Now we move to  $G$ corresponding to $J_r(D)$ where $D$ is an anisotropic quadratic space of dimension divisible by 4. 
We need the following facts, from \cite{Sahi_Tube}, the notation is taken from Section \ref{S:groups}. 
There is a maximal split torus $T_r$ in $G$ that 
gives rise to a restricted root system of type $C_r$.  
A maximal compact subgroup $K\subset G$ is the centralizer of an  involution given by conjugation action of 
\[ 
 \varphi \left(
\begin{matrix} 0 & - 1 \\ 
1 & 0 \end{matrix}
\right) 
\] 
where $\varphi : \SL_2 \rightarrow G$ arises from a set $S=\{ \beta_1, \ldots ,\beta_r\}$ of strongly orthogonal orthogonal roots. 
 Note that the matrix 
$\left(\begin{smallmatrix} 0 & -1 \\ 1& 0 \end{smallmatrix}\right)$ is conjugated to the matrix 
$\left(\begin{smallmatrix} i & 0 \\ 0& -i \end{smallmatrix}\right)$ in $\SL_2(\mathbb C)$ by the classical Cayley transform matrix. Since the centralizer of 
\[ 
 \varphi \left(
\begin{matrix} i & 0 \\ 
0 & -i \end{matrix}
\right) 
\] 
in $\mathfrak g_{\mathbb C}$ is $\mathfrak m_{\mathbb C}$ it follows that the Cayley transform conjugates $K_{\mathbb C}$ to $M_{\mathbb C}$. 
%Let $S=\{ \beta_1, \ldots ,\beta_r\}$ be the set strongly orthogonal orthogonal roots  which we naturally view as weights for $M_{\mathbb C}$. 
Let $\gamma_i$ be the weights for 
$K_{\mathbb C}$ obtained by transporting the weights $\beta_i$ by the Cayley transform. 
Now for every $r$-tuple of integers $a_1 \geq \ldots \geq a_r$ we have a $K$-type that corresponds to the irreducible 
representation of $K_{\mathbb C}$ with the highest weight 
\[ 
a_1 \gamma_1 + \cdots + a_r \gamma_r. 
\] 
These types appear in the degenerate principal series $I(s)$ with multiplicity one.  Sahi \cite{Sahi_Tube} has described the composition series of $I(s)$ as well as the 
Jantzen filtration which, in this case, simply measures the order of vanishing of the intertwining map $A(s) : I(s) \rightarrow I(-s)$ on each $K$-type. 
 We shall look only at the reducibility points 
 $1, 1+ d/2, \ldots ,1+ (r-1) d/2$. If we fix a reducibility point $1+ (i-1) d/2$, then the composition series of $I(s)$ has a shape of 
a truncated pyramid consisting of irreducible subquotients $V_{p,q}$ such that $p,q\geq 0$ and $r-i\leq p+q \leq r$. 
The subquotients of the Jantzen filtrations are the floors of the pyramid, in particular, they are 
isomorphic to direct sums of $V_{p,q}$ where $p+q=t$, a constant. The socle is the bottom floor i.e. $t=r$, and the co-socle is the top floor, i.e. $t=r-i$. 
The types of the  irreducible quotients $V_{p,q}$ of $I(1+ (i-1)d/2)$, in particular $p+q=r-i$, are particularly nice. They form a cone 
\[ 
(p-q)\frac{d}{4}(\gamma_1 + \ldots +\gamma_r) + (a_1 \gamma_1 + \cdots + a_r \gamma_r)
\] 
where  $a_{p+1}= \ldots = a_{n-q}=0$.

\section{Global results} 
In this section $k$ is a global field unless otherwise specified. 

\subsection{Global Fourier-Jacobi method} 
We follow here the work  of Ikeda \cite{Ikeda_FJ}. 
Let $Z\cong \mathbb G_a$  be the root subgroup corresponding to the highest root $\beta$.  
Recall that $Q=LV$ is the standard parabolic subgroup such that the Levi factor $L$ corresponds to the simple roots perpendicular to $\beta$. 
The unipotent radical $V$ is a Heisenberg group with the center $Z$.  Recall that $P=MN$ is a maximal parabolic subgroup in a standard position such that 
the unipotent radical $N$ is abelian.  In particular, $P$ contains $V$ and 
\[ 
V= (V\cap M) \cdot (V\cap N). 
\] 
One checks that $V \cap N$ is a maximal abelian subgroup of $V$. Write $Y=V\cap M$. Let $X$ be the unique abelian subgroup of $V$, normalized by the torus $T$, trivially intersecting $Z$, such that  $V\cap N= XZ$. Thus we can write $V=XYZ=YXZ$, where $XZ$ and $YZ$ are maximal abelian subgroups of $V$.  
Let $\psi$ denote a global or local, non-trivial additive character of $Z$.
The group commutator gives a $Z$-valued pairing between $X$ and $Y$. 
The pairing and $\psi$ define a Fourier transform from $S(X)$ and $S(Y)$ the spaces of Schwartz functions. Each of the two spaces realizes the Heisenberg 
representation of $V$, locally and globally. We shall use $S(X)$ unless specified otherwise. Let $G_1= [L,L]$ or an appropriate factor. 
 Let $J=G_1 V= V G_1$ be the Jacobi group. 
Then the Weil representation $\omega_{\psi}$  of $J$ is the unique extension of the Heisenberg representation of $V$ to $J$. 
Let $\mathbb A$ be the ring of adel\'es over $k$. Let 
$\Lambda$ be a functional on $S(X(\mathbb A))$ defined by 
\[ 
\Lambda(\phi)=\sum_{x\in X(k)} \phi(x) = \sum_{x\in X(k)} (\omega_{\psi}(x)(\phi))(0) 
\] 
for every $\phi\in S(X(\mathbb A))$. Now every $\phi$ defines an automorphic function $\Theta^{\phi}= \Lambda( \omega_{\psi}(g) \phi)$ on $J$. 
Let $f_s \in I(\chi,s)$ be a global holomorphic section. Then the Eisenstein series 
\begin{equation}
\label{defn_Eisenstein_series}
E(s)(g) = \sum_{\gamma\in P(k)\backslash G(k)} f_s(\gamma g) 
\end{equation} 
converges for $\Re(s)$ large enough. If $f$ is a smooth function on $Z(k)\backslash G(\mathbb A)$, define
\[ 
f_{\psi}(g) = \int_{Z(k)\backslash Z(\mathbb A)} f(zg) \bar \psi(z) ~dz. 
\] 
Write $G(k) =\cup_{w\in S} P(k) w Q(k)$, as a union of double cosets, and 
$E(s) =\sum_{w\in S} E(s)^{w}$ by breaking up the sum over individual cosets. It follows, from Lemma 4.2.2 in \cite{Weissman_small},  that $E(s)_{\psi}^w=0 $ except for $w=w_\beta$, representing the open double coset. 
It follows that, for $g_1\in G_1$, 
\[ 
\int_{V(k)\backslash V(\mathbb A)} E_{\psi}(v g_1) \overline{\Theta^{\phi}}(vg_1)  ~dv = 
\int_{V(k)\backslash V(\mathbb A)}  \sum_{\gamma\in P(k)\backslash P(k) w_{\beta} Q(k)} f_s(\gamma vg_1) \overline{\Theta^{\phi}}(vg_1) ~dv. 
\] 
Recall that $P_1=G_1\cap P$. It is easy to compute 
$w_{\beta}^{-1} P w_{\beta} \cap Q$ and verify that 
\[ 
P\backslash P w_{\beta} Q= w_{\beta} \cdot (Y Z\times P_1 \backslash G_1). 
\] 
Now the above integral can be written as 
\[ 
\int_{V(k)\backslash V(\mathbb A)}  \sum_{\gamma_2\in Y(k)Z(k)}\sum_{\gamma_1\in P_1(k)\backslash G_1(k)} f_s(w_{\beta} \gamma_2 \gamma_1  vg_1)
 \overline{\Theta^{\phi}}(\gamma_2 \gamma_1vg_1) ~dv. 
\] 
where we used $\Theta^{\phi}(vg_1)= \Theta^{\phi}(\gamma_2 \gamma_1vg_1)$. After the change of integration $v:= \gamma_1^{-1} v \gamma_1$, we can contract the integral and 
the first sum, giving 
\[ 
 \int_{X(k)\backslash V(\mathbb A)}  \sum_{\gamma_1\in P_1(k)\backslash G_1(k)} f_s(w_{\beta} v\gamma_1g_1) \overline{\Theta^{\phi}}(v\gamma_1 g_1) ~dv. 
\] 
Finally, using the definition $\Theta^{\phi}$, we arrive to 
\[
 \int_{V(\mathbb A)}  \sum_{\gamma_1\in P_1(k)\backslash G_1(k)} f_s(w_{\beta} v\gamma_1g_1) \overline{\omega_{\psi}(v\gamma_1 g_1)(\phi)(0)} ~dv. 
 \]
 Let, for $g_1\in G_1(\mathbb A)$, 
 \[ 
 F_{s}(g_1)  = \int_{V(\mathbb A)}   f_s(w_{\beta} vg_1) \overline{\omega_{\psi}(v g_1)(\phi)(0)} ~dv. 
 \] 
 Then $F_{s} \in I_1(\chi\chi_D,s)$, and we have shown that 
 \[ 
\int_{V(k)\backslash V(\mathbb A)} E_{\psi}(v g_1) \overline{\Theta^{\phi}}(vg_1) ~dv = 
 \sum_{\gamma_1\in P_1(k)\backslash G_1(k)} f_{1,s}(\gamma_1g_1) = E_F(s)(g_1)  
 \] 
 the Eisenstein series on $G_1$ attached to $F_{s}$. Of course, so far, this works  for $\Re(s)$ large enough. We shall now show that $F_s$  extends to a holomorphic section 
 for $\Re(s)>0$. Consider firstly the local situation. For $\Re(s)$ large enough, the local integral 
\[ 
F_s(g_1) = \int_{V}   f_s(w_{\beta} vg_1) \overline{\omega_{\psi}(v g_1)(\phi)(0)} ~dv 
 \] 
 defines an intertwining a map from $I(\chi,s) \otimes \omega_{\bar \psi}$ onto  $I_1(\chi\chi_D,s)$, intertwining the action of $J=G_1V$. 
 \medskip
 
 %The rest of this section has specific numerical data for $G$ of type $D_{2n}$ and $M$ in $P=MN$ of type $A_{2n-1}$. 

\begin{lem}
\label{local_FJ_integral}
Recall of the notation from section 2.1.  The local intertwining map $(f_s, \phi) \rightarrow F_s$ from $I(\chi,s) \otimes \omega_{\bar \psi}$ to  $I_1(\chi\chi_D,s)$ extends holomorphically to the region $\Re(s)>-(r-1)\frac{d}{2}.$ The map is non-zero for every $s$ in the region. 
\end{lem} 
\begin{proof}  
 We shall construct the continuation by writing the integral as an iterated integral, firstly integrating over $X$. We can assume that $g_1=1$, 
 and write $v=xyz$. Note that $w_{\beta}^{-1} x w_{\beta} \in P$, so $f_s(w_{\beta} xyz)=f_s(w_{\beta} yz)$. Since  
 $\omega_{\psi}(v )(\phi)(0)$ is equal to $\phi(x)\psi([y,-x])\psi(z)$, integrating over $X$ amounts to taking the Fourier transform of $\phi$, 
 evaluating at $y$. 
  Thus the local integral is equal to 
 \[ 
 \int_{YZ}   f_s(w_{\beta} yz) \bar{\psi}(z) \overline{\hat{\phi}(y)} ~dydz . 
 \] 
Next,  by an  easy $SL_2$-computation, the integral of $f_s(w_{\beta} yz)$  over $Z$ is absolutely converging if $\Re(s)>-(r-1)\frac{d}{2} $, and the output depends polynomially on $|y|$. Since 
$\hat\phi(y)$ is rapidly decreasing, it follows that the integral over $YZ$ is absolutely converging in the same range of $s$. By analytic continuation, 
the map $(f_s, \phi) \rightarrow F_s$ given by 
the iterated integral, intertwines the actions of $J=G_1V$, since it does so for large $\Re(s)$. The map is easily seen to be non-zero as $\hat\phi$ can be arbitrary. 
\end{proof} 

We continue in the setting of the previous lemma, and compute the local integral explicitly in the $p$-adic case, assuming that all data are unramified. 
In that case $\hat\phi$ is the characteristic function of  $Y(O)$, where $O$ is the ring of integers in the local field. 
Hence we can assume that $y\in Y(O)$, then $f_s(w_{\beta} yz)=f_s(w_{\beta} z)$ and the integral reduces to 
 \[ 
 \int_{Z}   f_s(w_{\beta} z) \bar{\psi}(z) dz = 1-\chi(\varpi)\frac{1}{q^{s+1+(r-1)\frac{d}{2}}}= L(s+1+(r-1)\frac{d}{2},\chi)^{-1}
 \] 
 where the quantity on the right hand side is obtained by a very easy $\SL_2$-computation (cf.~e.g.~\cite{Goldfeld_Hundley_book_1}, Proposition 1.6.5.). 
 Here $q$ is the order of the residue field of $O$ and ${\varpi}$ is the uniformizer.
 The product of these local factors is convergent and non-zero if $\Re(s)> -(r-1)\frac{d}{2}.$ Hence we have an analytic continuation of the global integral as well. 
 
\medskip
\begin{lem} \label{L:local_integral} 
Assume that  $\Re(s)>-(r-1)\frac{d}{2}$ and we are in a local, $p$-adic,  situation.  Assume that the assumptions of Theorems \ref{comp_series_r=2_p-adic} or \ref{comp_series_r>2_p-adic} are met.  Let $f_s\in I(\chi, s)$ and $\phi\in S(X)$. 
The map $(f_s, \phi) \rightarrow F_s$ from $I(\chi, s) \otimes \omega_{\bar \psi}$ to  $I_1(\chi \chi_D, s)$ is surjective.  Moreover, if $f_s$ is contained in the unique irreducible 
submodule of $I(\chi, s)$, then $F_s$ is contained in the unique irreducible submodule of $I_1(\chi\chi_D, s)$.
\end{lem} 
\begin{proof}  
We remark that Ikeda proves surjectivity by an explicit calculation. Here we present another argument in the $p$-adic case which gives the additional information about submodules. 
The map is clearly non-zero, and it intertwines the actions of $G_1V$ where  $V$ acts trivially on $I_1(\chi \chi_D,s)$.  In particular, $Z$ acts trivially on $I_1(\chi \chi_D,s)$  and 
the map descends to a non-zero map 
 from $I(\chi, s)_{Z,\psi}\otimes \omega_{\bar\psi}$ to $I_1(\chi \chi_D, s).$ 
 By Proposition 3.2. of \cite{Weissman_small} and Theorem \ref{FJ_local_degenerate_principal}, $I(\chi, s)_{Z,\psi} \cong I_1(\chi\chi_D, s) \otimes \omega_{\psi}$, as $G_1V$-module. 
  Since $V$ acts trivially on $I_1(\chi\chi_D, s)$, the map descends to a map from $I_1(\chi\chi_D, s) \otimes (\omega_{\psi}\otimes\omega_{\bar\psi})_V$ to $I_1(\chi\chi_D, s)$ intertwining the actions of $G_1$. 
Here $(\omega_{\psi}\otimes\omega_{\bar\psi})_V$ denotes the maximal quotient on which $V$ acts trivially. It is one-dimensional by Schur's lemma. 
Hence we get a non-trivial map from $I_1(\chi\chi_D, s)$ to $I_1(\chi\chi_D, s)$.  By our assumption on $G$, $I_1(s)$, if reducible, is of length 2, multiplicity free 
 and indecomposable, this map must be a multiple of the identity map.  From this description of the map $(f_s, \phi) \rightarrow F_s$ it is straightforward to check the lemma. 
\end{proof}

 \begin{prop}
\label{possible_pole_Eisenstein} Assume $G$ arises from a Jordan algebra $J_r(D)$. Let $d=\dim D$. 
   Let $\chi$  be a Grossencharacter satisfying $\chi^2=1$. Let $E(s)$ be the Eisenstein series 
  arising from  a holomorphic section $f_s$ of the principal series $I(\chi,s)$. 
\begin{enumerate} 
\item  Assume that $\chi_D=1$. Then $E(s)$ is holomorphic at $s_0>0$ except, possibly, when 
$\chi=1$ and  $s_0=  1, 1+d/2, \ldots, 1+ (r-1) d/2$. 

  \item Assume $r=2.$ 
\begin{enumerate}
\item Assume $d=2n-1$.  Then $E(s)$ is holomorphic at $s_0>0$ except, possibly, when $s_0=1/2$, or $\chi=1$ and $s_0=n+\frac{1}{2}$, where 
the trivial representation is the residue. 
  \item Assume $d=2n-2$. Let $\chi_D$ be the Grossencharacter attached globally to $D.$ 
  Then $E(s)$ is holomorphic at $s_0>0$ except, possibly, when $\chi=\chi_D$ and $s_0=1$, or $\chi=1$ and $s_0=n$, where the trivial representation is the residue. 
  \end{enumerate}
 \end{enumerate}
 The possible poles are at most simple.  Moreover, if the local component of $f_{s_0}$ at any $p$-adic place is contained in the unique submodule of $I(\chi, s_0)$, then $E(\chi,s)(f)$ is holomorphic at $s_0.$ 
 \end{prop} 
 \begin{proof} 
 For notational convenience we deal with the first case. Fix $\Re(s_0)>0$, and expand 
 \[ 
 E(s)(g) = \frac{f(g)}{(s-s_0)^l} + \text{ higher powers of } (s-s_0) 
 \] 
 where $f(g)$ is the residual form. 
 Then  for $s_0= 1+ (r-1) d/2$ and $\chi=1$ we have $l=1$ and the residual representation is the trivial representation; this is a well known case. 
 So assume that $s_0\ne 1+ (r-1) d/2$  or $\chi\neq 1$. Let 
 $\mathcal A$ be the residual automorphic representation.  (We work with spaces of $K$-finite functions.) 
 We claim that there exists $f\in \mathcal A$ such that the global Fourier coefficient $f_{\psi}(1)$ is non-zero. 
 Assume that $f_{\psi}(g)=0$ for all non-trivial characters $\psi$ and $g\in G(\mathbb A)$. 
 Then $f$ is left $Z(\mathbb A)$-invariant. Let $v$ be a local place and $\mathbb A_v$ the ring of adel\'es with the local factor $k_v$ removed. 
  By the weak approximation theorem, $f$ is determined by its restriction to $G(\mathbb A_v)$. Since $G(\mathbb A_v)$ and $Z(k_v)$ commute, it follows that 
 $f$ is left $Z(k_v)$-invariant. If this is true for every $f$, then the $v$-adic component of $\mathcal A$ is the trivial representation, a contradiction to the assumption on $s$. 
 Hence there exists $f\in \mathcal A$ and $g\in G(\mathbb A)$ such that $f_{\psi}(g)\neq 0.$  We write $g=g_{\infty}g_f,$ where $g_{\infty}$ denotes the archimedean part, and $g_f$ the part belonging to the finite adeles. We easily get rid of the $g_f$--part by a right translation, so we can assume that $f_{\psi}(g)\neq 0$ for $g\in G(\mathbb A_{\infty})$ (i.e. $G(\R),$ if we are working over $\Q$). Since $f$ is $K$--finite it is analytic, and then $f_{\psi}$ is analytic as well. So we expand  this non-trivial $f_{\psi}$ near identity, and there exists an element of the universal enveloping algebra, say $\mathfrak D$, such that $\mathfrak D(f_{\psi})(1)\neq 0,$ 
 but $\mathfrak D(f_{\psi})(1)=(\mathfrak Df)_{\psi}(1),$ so we have found  $h\in \mathcal A$ such that $h_{\psi}(1)\neq 0.$

 So let $f\in \mathcal A$ such that $f_{\psi}(1)\neq 0$. There exists $\phi\in S(X(\mathbb A))$ such that 
 \[ 
g_1\mapsto \int_{V(k)\backslash V(\mathbb A)} f_{\psi}(v g_1) \overline{\Theta^{\phi}}(vg_1)  ~dv
\] 
is a non-trivial function on $G_1$. It follows that $E_F(s)$ has a pole of order $l$. By the induction assumption $l =0$ or $1$ and $l=1$ only if $s_0$ is one of the 
listed values. Furthermore, by the induction assumption and Lemma \ref{L:local_integral}, $E(s)$ has no pole if the local component of $f_{s_0}$ is in the irreducible submodule of $I(s_0)$. 
 \end{proof}

 \subsection{Main result} 
 
To prove existence of the poles i.e.~the if and only if result, once can argue as Ikeda and prove that the local integral is surjective at the archimedean places. 
Instead we shall compute the constant term of the Eisenstein series along the unipotent radical of the minimal parabolic. 
The full constant terms involves a complicated sum over the Weyl group, however, we shall look only the summand where the intertwining operator $A(s)$ appears.  
This will give us not only existence of the pole, but also 
 a control of the structure of the residual representation. In order to keep arguments as simple as possible, we shall henceforth work with 
$J_r(D)$ such that $d\equiv 0 \pmod{4}$ and $D$ has trivial discriminant i.e. $\chi_D=1$. Then the above  result simply says that 
 $E(s)$ has possible simple poles at odd integers $1, 1+ d/2, \ldots, 1 + (r-1) d/2$.

\begin{thm} Assume $G$ corresponds to $J_r(D)$ such that $d\equiv 0 \pmod{4}$. In addition, assume that: 
\begin{enumerate} 
\item The discriminant of the quadratic space $D$ is trivial, i.e. $\chi_D=1$. 
\item The quadratic space $D$ is either split or totally anisotropic. 
\item For every real place $v$, $D_v$ is either split or totally anisotropic. 
\end{enumerate} 
Then the Eisenstein series $E(s)$ has simple poles at $s_0=1, 1+ d/2, \ldots, 1 + (r-1) d/2$.  At each $s_0$ the residual representation is square integrable and 
isomorphic to the co-socle of the global degenerate principal series $I(s_0)$. 
\end{thm}

Observe that the conditions 1) - 3) are automatically satisfied if $r\geq 3$. 
\begin{proof}  

Let $E(s)$ be the Eisenstein series attached to a holomorphic section $f(s)= \otimes_v f_v(s)$ in $I(s)=\otimes_v I_v(s)$.  Let $s_0$ be one of the points, and $v$ a
$p$-adic place where $G$ is split. ($G$ is split at almost all primes, as we shall argue in a moment.)  If $f_v(s_0)$ belongs to the irreducible submodule of 
$I_v(s_0)$ then $E(s)$ is holomorphic at $s_0$.  In particular, only the irreducible spherical quotient at the place $v$ can contribute to the residual representation. 
By Proposition  \ref{P:aubert_dual}  the spherical quotient of $I_v(s_0)$ is the Aubert dual of a square integrable representation. It follows that 
the residual representation is square integrable. Hence it decomposes as a direct sum of irreducible representations, so it must be a quotient of the 
co-socle of $I(s_0)$. In order to show that the residual representation is the full co-socle we need to show that the pole is achieved as the section 
$f(s)$ passes through types belonging to irreducible representations in the co-socle. 

\smallskip 

Let $\Phi$ denote the root system of $G$ relative to a maximal split torus. If $D$ is split then $G$ is split (Chevalley) group, if $D$ is anisotropic, 
then the maximal split torus we can take $T_r$ as in Section  \ref{S:groups} . 
 Let $W$ be the corresponding Weyl group. Let $P_0=M_0N_0$ be a minimal parabolic subgroup containing the split torus, corresponding to a choice of positive roots $\Phi^+$ 
 in $\Phi$, and we can assume that the parabolic group  $P=MN$ in the standard position i.e. $M_0\subseteq M$ and $N\subseteq N_0$.
  Let $\Phi_M^+ \subseteq \Phi^+$ be the positive roots for $M$.  Let 
 \[ 
W(M) =\{ w\in W:w(\Phi_M^+)>0\}.
\] 
The element $w_0$, the product of the longest Weyl group elements for $G$ and $M$,  belongs to $W(M)$. 
We shall use that $w_0$ permutes $\Phi_M^+$ and that $w_0(\omega)=\omega^{-1}$. 
The degenerate principal series $I(s)=\Ind_P^G(|\omega |^s)$ is naturally embedded in the principal series
  $\Ind_{P_0}^G(\chi_s)$ where $\chi_s$ is a character of $M_0$.  Just mentioned properties of $w_0$ imply that $w_0(\chi_s)=\chi_{-s}$. 
  If $E(s)$ is the Eisenstein series built from a holomorphic section $f(s)$, 
  its constant term along $N_0$ is naturally a function on $M_0$. As such, it is a sum 
  \[ 
 \sum_{w\in W(M)}  d_w(s)w(\chi_s) 
 \] 
 where $d_w(s)$ are meromorphic functions that depend on $f(s)$.  We look at the summand corresponding to $w_0$. 
 Assume firstly that $G$ is split.  Let $f_v(s)$ be the normalized spherical vector in 
 the local principal series representations  $I_v(s)$. If $s>0$ then $I_v(s)$ is generated by $f_v(s)$. Let $E(s)$ be the Eisenstein series 
 corresponding to $f(s)= \otimes_v f_v(s)$.  The contribution of $w_0$ to the constant term is the restriction to $M_0$ of 
 $A(s)(f(s)) = c(s) f(-s)$, where, by Lemma  \ref{L:local_c} 
 \[ 
c(s) =\prod_{i=0}^{r-1} \frac{\zeta(s- i d/2)}{\zeta(s +i d/2 +1)}. 
 \] 
 Here $\zeta(s)$ is the global Dedekind $\zeta$-function corresponding to the the number field $k$. It is well known that $\zeta(s)$ has a simple pole at $s=1$, hence 
  $d_{w_0}(s)=c(s)$ has simple poles at the points of interest. We now look at the case of anisotropic $D$. 
  
  \begin{lem}  For almost all places $v$, the quadratic space $D_v$ is split.
  \end{lem} 
  \begin{proof} 
   Assume that $v$ is a $p$-adic place. Since the discriminant of $D_v$ is trivial,  $D_v$ is either split or it has a 4 dimensional anisotropic kernel isomorphic to 
   a quaternion algebra. The
  isomorphism class is determined by the isomorphism class of the Clifford algebra attached to $D_v$. But this algebra is a localization of the global Clifford algebra attached to 
  $D$. The Clifford algebra is a central simple algebra, it localizes to a matrix algebra for almost all places. This proves the lemma. 
  \end{proof} 
  
  \smallskip 
  So let $S$ be the finite set of places such that $D_v$ is split for $v\notin S$. We consider the Eisenstein series $E(s)$ corresponding to the constant 
  section $f(s)= \otimes_v f_v(s)$ where, for all $v\not\in S$, $f_v(s)$ is the spherical vector, while for $v\in S$, $f_v(s)$ is arbitrary. Then the contribution of $w_0$ 
  to the constant term is again given by $A(s) f(s)$. Since we know how to compute the action of the intertwining operator on the spherical vector 
  at the places $v\notin S$, we have 
  \[ 
  A(s) f(s) = c(s) (\otimes_{v\in S} c_v(s)^{-1} A_v(s) f_v(s)) \otimes (\otimes_{v\notin S} f_v(-s)). 
  \] 
  Since $c_v(s)^{-1}$ are non-zero by Corollary \ref{C:non-vanishing_split} and the local intertwining operators are always non-zero, we see $d_{w_0}(s)$ has poles at the points of interest.  
  In fact, since the holomorphic properties of $A_v(s)$ reflect the Jantzen filtration, we see that the pole is achieved for $f_v$ in any $K$-type belonging 
  to the co-socle of $I_v(s)$.

  \smallskip 
  It remains to show that the pole, at the point $s_0$, of the $w_0$-summand in the constant term is not cancelled out by a pole of the $w$-summand for some other $w\in W(M)$. 
  The cancellation can happen only if 
  \[ 
  w(\chi_{s_0})= w_0(\chi_{s_0}). 
  \] 
  Thus we need to show that there is no such $w\in W(M)$. This is an easy check left to the reader in the case of split groups, but we provide details if 
  $D$ is totally anisotropic. In this case $\Phi$ is of the type $C_r$ and $\Phi_M$ of the type $A_{r-1}$.   Any real character of $M_0$ is determined 
   by the restriction to the maximal split torus $T_r$.  Recall that any element in $T_r$  
  is uniquely written as a product of $\omega_i(t_i)$ where $\omega_i^{\vee}$ are the co-characters defined by the equation (\ref{E:co-character}). 
  Thus any real character $\chi$ of $T_r$ is determined 
  by an $r$-tuple $(s_1, \ldots, s_r)$ of real numbers defined by $\chi(\omega_i)= |\cdot |^{s_i}.$ In these coordinates the modular character is   
\[
\rho=(1+ (r-1)d, \ldots , 1 +d,1). 
\] 
Note that the difference between the consecutive entires is $d$, which reflects the fact that short root spaces are $d$-dimensional. In order to compute 
$\chi_s$ we observe that the (group) root spaces corresponding to $\pm\alpha$ where $\alpha$ is a short simple root (i.e.~a simple root of $M$) 
generate a group isomorphic to $\Spin(H\oplus D)$ where 
$H$ is a 2-dimensional hyperbolic plane. The degenerate principal series for this group, with respect to the maximal parabolic subgroup whose unipotent radical 
is the root space of $\alpha$, contains the trivial representation as a submodule for $s=-d$. It follows that $s_{i}-s_{i+1}=-d$ for the coordinates of $\chi_s$. These 
equations pin down a line, and the linear parameter $s$ is fixed by demanding that $w_0(\chi_s)=\chi_{-s}$ and $\chi_{s_0}=-\rho$ for $s_0=-1-(r-1)d/2$. (At this point 
both series of representations contain the trivial representation as a submodule.) Putting everything together yields 
\[ 
\chi_s = (s,s,\ldots , s) + \frac{d}{2}(1-r, 3-r, \ldots , r-1). 
\] 
We claim that $\chi_{s_0}$ is regular at the reducibility points.  
To that end, recall that a character $\chi=(s_1, \ldots, s_r)$  is singular if it is contained in a wall 
  $s_i=0$, $s_i-s_j=$ or $s_i+s_j=0$.   Since the coordinates of $\chi_{s_0}$ form a strictly increasing sequence of odd integers, it is clear that $\chi_{s_0}$ cannot 
  satisfy the first two equations. Since $d$ is divisible by 4, the coordinates of $\chi_{s_0}$ are congruent modulo 4. The equation $s_i+s_j=0$ implies that $s_i$ and 
  $s_j$ are opposite integers. But two opposite odd integers are never congruent modulo 4,  hence $s_i+s_j=0$ cannot hold.  

\end{proof} 

If  $k=\mathbb Q$, and $J=J_3(\mathbb O)$ where $\mathbb O$ is the Cayley-Graves octonion algebra,
 then the residual representation at $s=5$ and $s=1$ contains the singular modular 
 form on the exceptional tube domain of the weight 4 and 8, respectively, discovered by Kim \cite{Kim_octonions}. 

\section{Acknowledgments} 
The first named author was supported in part by a Croatian Science Foundation grant no. 9364.
The second named author was supported in part by an NSF grant DMS-1359774.

\bibliographystyle{siam}
\bibliography{deg_eisen_Sp4_27_07}
\bigskip
\end{document}